\documentclass[pdflatex,sn-mathphys-num]{sn-jnl}

\usepackage{graphicx}%
\usepackage{multirow}%
\usepackage{amsmath,amssymb,amsfonts}%
\usepackage{amsthm}%
\usepackage{mathrsfs}%
\usepackage[title]{appendix}%
\usepackage{xcolor}%
\usepackage{textcomp}%
\usepackage{manyfoot}%
\usepackage{booktabs}%
\usepackage{algorithm}%
\usepackage{algorithmicx}%
\usepackage{algpseudocode}%
\usepackage{listings}%

\theoremstyle{thmstyleone}%
\newtheorem{theorem}{Theorem}
\newtheorem{proposition}[theorem]{Proposition}%
\newtheorem{conjecture}[theorem]{Conjecture}
\newtheorem{corollary}[theorem]{Corollary}
\newtheorem{lemma}[theorem]{Lemma}

\theoremstyle{thmstyletwo}%
\newtheorem{example}{Example}%
\newtheorem{remark}{Remark}%

\theoremstyle{thmstylethree}%
\newtheorem{definition}{Definition}%

\newcommand{\bb}[1]{\mathbb{#1}}
\newcommand{\conv}{\operatorname{conv}}

\begin{document}

\title[Sets Reconstructable with Medial Axis]{Sets Reconstructable with Medial Axis}


\author{\fnm{Adam} \sur{Białożyt}}\email{bialozyt@agh.edu.pl}


\affil{\orgdiv{Faculty of Applied Mathematics}, \orgname{AGH University of Krak\'ow}, \orgaddress{\street{al. Mickiewicza 30}, \city{Krak\'ow}, \postcode{30-059}, \country{Poland}}}

\abstract{The medial axis of a closed set is well established tool in pattern recognition, cherished for its power of reconstruction of shapes. Yet, a precise description of what the shape means in this context is missing in the literature. In this article we fill this gap answering the question which sets precisely are reconstructible from the medial axis information. We introduce the notion of a reconstructible point, being the point in the union of open balls of maximal radius, centred at points of the medial axis. The reconstructible points turn to be precisely the ones lying in the convex hull of the initial set and ones in the half-spaces supported by hyper-planes for which the convex hull and set have different intersection. We do not assume any additional structure of considered sets besides them being closed in $n-$dimensional Euclidean space. }



\maketitle

\section{Introduction}\label{sec1}
The idea of reducing a data set prior to analysis is ubiquitous in science. A celebrated family of such preprocessing techniques is the so-called skeletonization of a set.

Generally speaking, a vast array of different types of skeletal structures are used in data processing and pattern recognition. One of the commonly renowned systems of thinning is the medial axis transform. The medial axis of a closed set $X$ is a collection of points in the ambient space with more than one closest point in $X$ (see Definition~\ref{Def.MedialAxis} below and the example in Fig.~\ref{Fig.Szkielet}). Although the medial axis and other skeletal structures appear in the works of other mathematicians (Erdos, Federer, ... \cite{Erdos,Federer}), Harry Blum is typically considered to be their originator. In \cite{Blum} he argues for their feasibility in the fields of pattern recognition. One of the reasons why the medial axis gained its importance is because it properly encodes the shape of $X$. The medial axis is a strong topological retract and preserves the homotopy groups of $X^c$.


Common knowledge says that the medial axis of $X$ paired with the distance function allows one to reconstruct shapes. However, What does the shape mean here? And are we really sure we can reconstruct any? These questions appear to be swept under the rug. In this article, we will look at this question in more detail.

The article is organised as follows. In section 2 we provide the setting and notations of our investigations, the definition of the medial axis, and some similar notions. Next, in section 3, we discuss the reconstructability of points in $\bb{R}^n$ and present an example that requires for analysis all the presented theorems. In section~\ref{Sec.Stability} we quickly comment on the stability of reconstructible points. Lastly, we briefly discuss the possibility of adaptation of the results to other contexts and problems that might arise during that process.

\section{Basic definitions}\label{sec2}
Throughout the article, we assume that $X$ is a non-empty closed subset of $\mathbb{R}^n$ endowed with the Euclidean distance. Since we will use the symbol $d(\cdot, \cdot)$ to denote the distance to the set, we will stick to the norm notation $\|a-x\|$ to signify the distance of points $a$ and $x$. By $[a,x]$ (respectively, by $]a,x]$ and $]a,x[$) we denote the closed segment that joins $x$ and $a$ (resp. the segments $[a,x]\backslash \{a\}$ and $[a,x]\backslash \{a,x\}$).  

We will denote by $\bb{B}(a,r)$ an open ball of radius $r>0$ which is centred at point $a\in \mathbb{R}^n$; by $\bb{S}(a,r)$ we will denote the boundary of $\bb{B}(a,r)$, that is, the $(n-1)$-dimensional sphere of the same radius and co-centric with $\bb{B}(a,r)$.

Whenever in the paper the continuity (or upper and lower limits) of a family of sets or of a (multi)function is mentioned, it refers to the continuity (or upper and lower limits) in the Kuratowski sense. Let us quickly recall that given a family $\lbrace X_t\rbrace_{t\in\mathbb{R}^k}$ of subsets of $\mathbb{R}^n$, and a point $x\in \mathbb{R}^n$:
\[x\in\limsup_{t\rightarrow t_0} X_t\text{ iff there exist sequences }\mathbb{R}^k\ni t_\nu\rightarrow t_0\text{ and }X_{t_\nu}\ni x_\nu\rightarrow x;\]
\[x\in\liminf_{t\rightarrow t_0} X_t\text{ iff for each sequence }\mathbb{R}^k\ni t_\nu\rightarrow t_0\text{ we can find }X_{t_\nu}\ni x_\nu\rightarrow x.\]
Naturally, whenever two partial limits coincide, we call this set the Kuratowski limit. More on the Kuratowski convergence is found in \cite{RockafellarWets}, and an introduction to its relation with medial axes and conflict sets is given in \cite{BiaDenk_Kuratowski}.

\subsection{Medial axis and other skeletal structures}
Despite the popularity of the skeleton-like structure concepts, there is a bit of confusion as to what we should really call the medial axis. The literature varies when it comes to the details of the definition (or sometimes does not provide it at all). These details give rise to similar, however, not exactly the same objects.

To make things precise for us, let us call the medial axis nothing else but the locus of points in $\mathbb{R}^n$ that have more than one closest point in $X$. We will adopt the following notation. For a non-empty closed set $X$ and a point $a\in \mathbb{R}^n$ we denote by $d(a,X)$ the Euclidean distance from $a$ to $X$, that is, the infimum of $\|a-x\|$ taken over all points in $X$
\[d(a,X)=d(a):=\inf_{x\in X}\|a-x\|.\] 
Next, we collect the points of $X$ closest to $a\in \bb{R}^n$ in a set $m(a)$. 
\[
m_X(a) = m(a):=\{x\in X\mid \|a-x\|=d(a,X)\}\footnote{Even though, $m(a)$ is formally a set of points, we often identify it with a point of $\bb{R}^n$ when $\# m(a)=1$}.
\]
\begin{definition}\label{Def.MedialAxis}
    The medial axis of $X$ denoted by $M_X$ is the set of points $a\in \mathbb{R}^n$ for which $m(a)$ has more than one point.
    \[
    M_X:=\{a\in \mathbb{R}^n\mid \#m(a)>1\}.
    \]
\end{definition}

\begin{figure}
\centering
\includegraphics[width = 0.4\textwidth]{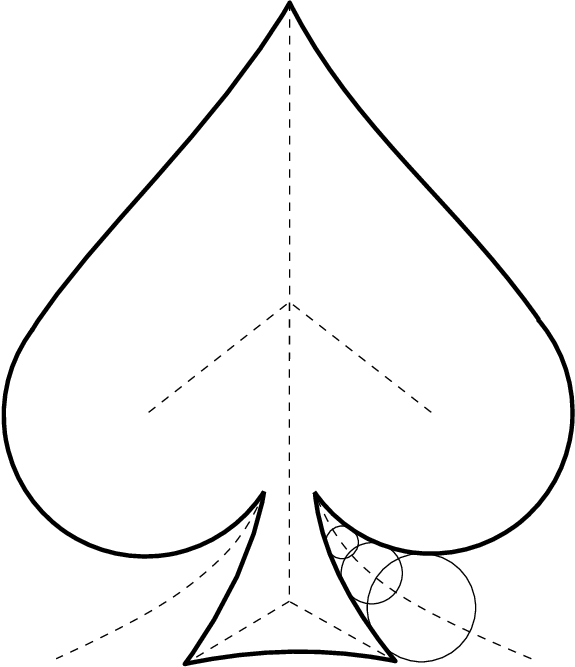}\label{Fig.Szkielet}
\caption{A medial axis (in a dashed line) of a spade shape. Three maximal balls with centres at points of the medial axis are drawn at the bottom right }
\end{figure}
Observe that adopting this point of view has consequences in $M_X\cap X=\emptyset$. The medial axis forms sort of an exoskeleton of a closed set.

\subsubsection{Central set}
A notion used simultaneously with the medial axis is the one of the central set. Its precise construction is as follows. For a closed set $X\subset\bb{R}^n$ consider a family \[\mathcal{B}_X:=\{\bb{B}(a,r)\mid a\in\mathbb{R}^n,\, r>0,\, \bb{B}(a,r)\subset X^c\}.\] Naturally, the family $\mathcal {B}_X$ is partially ordered by a relation of inclusion. We call a ball $\bb{B}(a,r)\in \mathcal{B}_X$ to be \textit{maximal} if it is a maximal element of the inclusion relation. That is, for any other ball $\bb{B}(a',r')\in \mathcal{B}_X$ we have $\bb{B}(a,r)\subset \bb{B}(a',r')\Rightarrow a=a',r=r'$.
\begin{definition}
    The central set of $X$ denoted by $C_X$ is the set of centres of maximal balls of the family $\mathcal{B}_X$.
\end{definition}

Even though the definition seems to construct the same object, in reality the two notions are subtly different. We only have $M_X\subset C_X\subset\overline{M_X}$ (see \cite{BirbrairDenkowski,BirbrairDenkowskiErratum}) and either of the inclusions can be strict. In practice, finding the points of $C_X$ is often easier than finding the ones of $M_X$. Luckily, due to inclusions, if we localise a point $c\in C_X$, we can be sure that $a\in M_X$ can be found arbitrarily close to $c$. For a more detailed study of the differences between sets $C_X$ and $M_X$, we refer to \cite{Fremlin}.

\subsubsection[Lambda medial axis]{$\lambda$-medial axis and scale medial axis}\label{Subsect.LambdaMedialAxisDef}
One of the hardships of applying the medial axis transformation is the sensitivity of the latter on the perturbation in the data. Even a small distortion of a set can result in dramatic changes in the medial axis. One of the approaches to tempering the process involves a change in the measurement of the distortion in the medial axis\footnote{Rather than measuring the change of $M_X$ and $M_{x'}$ with the Hausdorff distance $d_H(M_X,M_{X'})$, we measure it only with $\sup_{x\in M_X} d(x,M_{X'})$.} (see \cite{G._Matheron._Examples_of_topological_properties_of_skeletons.}). Another one treats the distortion of the set as a continuous process\footnote{ We treat $X'$ as a Kuratowski limit $\lim X_t$ of a family $\{X_t\}$ with $X_0=X$. Then the inner-continuity of the medial axes occurs $M_{X'}\subset \liminf M_{X_t}$.}
(see \cite{BiaDenk_Kuratowski}). 

A rather well studied approach calls for pruning of the medial axis. We refer here to the so-called $\lambda-$medial axis $M_\lambda$, introduced by Chazal and Lieutier \cite{ChazalLieutier}. The $\lambda-$medial axis is the result of removing from $M_X$ the points $a$ with $m(a)$ that fit within some ball $\overline{\bb{B}(p,\lambda)}$. The $\lambda-$medial axis is stable under small perturbations of the set $X$. In general, this approach to stability comes with a price. The $\lambda-$medial axis might not preserve the homotopy type of $X^c$ (so the set reconstructed with $M_\lambda (X)$ might have a different homotopy type than $X^c$). In particular, Chazal and Lieutier showed that the homotopy type is preserved when the so-called \textit{weak feature size } of $X$ is greater than $\lambda$. In real life applications, this condition is often satisfied, so the homotopy type of $X^c$ is most often preserved.


\subsubsection{Medial axis on Riemannian manifolds}
The question of the number of points that realise the distance $d(a)$ also makes sense in the setting of Riemannian manifolds. Changing the surrounding, an interesting phenomenon emerges; one  that was never met in $\bb{R}^n$. Two points on the manifold might be connected in the shortest way by two different paths. The medial axis of a closed set $X$ on a complete Riemannian manifold is adapted to detect this phenomenon. It collects points that have non-unique shortest paths (geodesics) that connect to $X$. In this way, it readily generalises the notion of a cut locus of a point. 

The medial axis on the Riemannian manifolds \cite{Albano2, Houston, Giesen, Wolter}

\section{Reconstructible sets}\label{sec3}
The reconstruction process of $X^c:=\bb{R}^n\backslash X$ using the medial axis tranformation consists of taking a union of open balls centred at points $a\in M_X$ and of radius $d(a)$ (see Fig.~\ref{Fig.RekonElipsa}). To make our study more structured, we pose the following definition.

\begin{figure}
\centering
\includegraphics[width=0.7\textwidth]{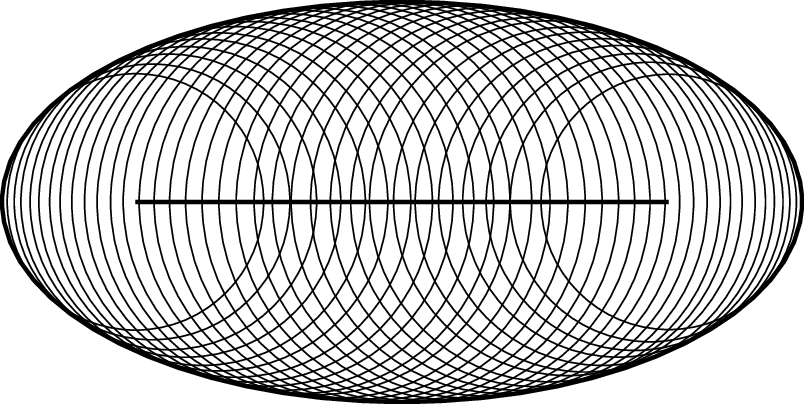}
\label{Fig.RekonElipsa}
\caption{The reconstruction process of the interior of an ellipse.}
\end{figure}

\begin{definition}\label{DefOfReconstructible}
    Let $X\subset \bb{R}^n$ be a closed set. We call a point $p\in X^c$ \textit{reconstructible} if $p\in \bigcup_{a\in M_X}\bb{B}(a,d(a))$. We call the set $X^c$ \textit{reconstructible} if every point $p\in X^c$ is reconstructible.    
\end{definition}

If the context is clear, we will omit the set $X$ used to establish the reconstructibility of a point. In case such precision is necessary, we will signal $X$ by noting the complement that contains the point $p$ and write $p\in X^c$. 
Let us quickly observe that, without any change in the notion of a reconstructible point, we can use the central set $C_X$ in \textit{lieu} of the medial axis $M_X$ in Definition~\ref{DefOfReconstructible}.

\begin{lemma}\label{Lem.PrimaryLemma}
    Let $X\subset \bb{R}^n$ be a closed set and let $p\in X^c$. Then there exists $c\in C_X$ such that $p\in \bb{B}(c,d(c))$ if and only if $p$ is reconstructable.  
\end{lemma}
\begin{proof}
    If $p$ is reconstructible, then there exist $a\in M_X$ such that $p\in \bb{B}(a,d(a))$. But we have the inclusion $M_X\subset C_X$, so $a\in C_X$. 
    
    On the other hand, if there exists $c\in C_X$ such that $p\in \bb{B}(c,d(c))$ then $\|p-c\|<d(c)$. The point $c$ can be approximated by a sequence of points $x_n\in M_X$. Since the distance functions (both the point-point and the point-set one) are continuous, for $n$ large enough, we have $\|p-x_n\|<d(x_n)$. For such $n$ we have $x_n\in M_X$ and $p\in \bb{B}(x_n,d(x_n))$ so that the point $p$ is reconstructible.
\end{proof}

Now, our main question in this article is which sets are reconstructible with the help of the medial axis. The question might seem almost blasphemous. After all, the medial axis transformation is renown for reconstructing shapes \cite{Attali2009, Chang2001, GiblinReconstruction, Younes2019}. However, a short look at the example of a closed unit disc $\bb{D}:=\overline{\bb{B}(0,1)}$ (or, more generally, any convex set) shows us with great clarity that, in fact, not all points of $\bb{D}^c$ can be retrieved from $M_\bb{D}$. Indeed, the medial axis $M_\bb{D}$ is empty! So, there is no point that can act as a centre of the ball in the union $\bigcup_{a\in M_\bb{D}}\bb{B}(a,d(a))$ and (if we skip the ambiguity of the set theoretic union over an empty family) the reconstruction process fails to bring back any of the points of $\bb{D}^c$.

\begin{proposition}\label{Prop.ConvexHull}
    Let $X\subset \bb{R}^n$ be a closed set and let $p\in X^c$ be a point of the convex hull $\conv X$. Then $p$ is reconstructible. 
\end{proposition}
\begin{proof}
    Due to the Caratheodory lemma, $p$ is an interior point of at most $n-$dimensional simplex with vertices $v_0,\ldots,v_k\in X$. Assume that $p\notin M_X$ (otherwise, $p$ is surely reconstructable). Thus, we can find a unique closest point of $p$ in $X$ that we denote by $x_p$.

    The idea now is to show that we cannot inflate the ball $\bb{B}(p,d(p))$ indefinitely while maintaining tangency with $X$ only at $x_p$. Denoting by $p_t: = t(p-x_p)+x_p$ we can rephrase that claim more formally as a condition
    \[r:=\sup\{t>0\mid \bb{B}(p_t,d(p_t))\cap X = \emptyset\} <\infty.\]
    
    Observe that $r=\infty$ means in fact that the whole half-space $\{x\in\bb{R}^n\mid\langle x-x_p,p-x_p\rangle>0\}$ is disjoint with $X$, as it is the union of the family of balls $\bb{B}(p_t,d(p_t))$ with $t>0$. We then have $\langle v_i-x_p,p-x_p\rangle\leq 0$ for $i = 0,\ldots,k$. Recall now that $p$ was the interior point of the simplex spanned by $v_0,\ldots,v_k$. Therefore, we can represent $p$ as a sum $\sum \lambda_iv_i$ with $\sum\lambda_i=1,\;\lambda_i\in (0,1)$. However, that leads to a contradiction
    \[0<\langle p-x_p,p-x_p\rangle = \langle \sum_{i=0}^k \lambda_iv_i-\sum_{i=0}^k\lambda_ix_p,p-x_p\rangle= \sum_{i=0}^k \lambda_i\langle v_i-x_p,p-x_p\rangle\leq 0.\]
    
    Finally, as $r<0$ the point $p_r$ is well defined and is a centre of a maximal ball in $X^c$. In other words $p_r\in C_X$. Due to Lemma~\ref{Lem.PrimaryLemma} the point $p$ is reconstructible. 
\end{proof}

\begin{remark}
    Proposition~\ref{Prop.ConvexHull} fully encapsulates the classical use of the medial axis, i.e. the reconstruction of $X^c$ which are bounded domains. Indeed, when $X^c$ is bounded, we can find an open ball $\bb{B}(a,r)$ such that $X^c\subset \bb{B}(a,r)$. What follows is $\bb{S}(a,r)\subset (X^c)^c=X$ and \[X^c\cap \conv X\subset X^c\cap \conv\bb{S}(a,r)=X^c.\] Therefore, the whole set $X^c$ is reconstructible due to the last proposition.
\end{remark}

Having dealt with the points of $\conv X$ we wish to go further; beyond the convex hull. To analyse the points of $X^c$ outside the convex hull of $X$ we will still lean on $\conv X$. A context suitable to formalise our investigations is provided by a notion that comes from convex geometry.

\begin{definition}
    Let $X\subset \bb{R}^n$ be a set (not necessarily closed). We call a hyperplane $L\subset\bb{R}^n$ \textit{a supporting hyperplane of $ X$} if $X \cap L\neq \emptyset$ and $X$ is entirely contained in one of the two closed half-spaces bounded by $L$. If that is the case, we denote by $L^+$ the open half-space disjoint with~$X$.  
\end{definition}

It is a standard fact from convex geometry that any point on the boundary of a closed convex set $X$ belongs to some supporting hyperplane of $X$. However, this hyperplane might not be unique. Consider a cone $C=\{(x,y)\in\bb{R}^2\mid y\geq |x|\}$. Then any plane $L_a=\{(x,y)\in\bb{R}^2\mid y=ax\}$ with $|a|\leq 1$ is a supporting plane of $C$ with $0\in L_a$. If $X$ is not closed, the supporting plane might or might not exist. To witness a convex set with no supporting planes at all, let us take the open unit ball $\bb{B}(0,1)$. An example of non closed set with supporting hyperplanes at every point of the boundary is given by $\{(x,y)\in\bb{R}^2\mid y<0\}\cup\{(0,0)\}$.

Moreover if $X$ is closed and convex, then it is the intersection of the family $\{\bb{R}^n\backslash L^+\}$ where $L$ are the supporting hyperplanes of $X$.

\begin{proposition}\label{Prop.Polprzestrzenie}
    Let $X\subset \bb{R}^n$ be a closet set and let $p\in X^c$. Let $L$ be a supporting plane of $\conv X$ such that $X \cap L$ is non-convex. 
    If $p\in L^+$ then $p$ is reconstructible.
\end{proposition}
\begin{proof}

    Since the intersection $X_L:= X \cap L$ is non-convex, we know that $M_{X_L}$ is non-empty. Now, 
    to simplify the exposition of the proof, assume, without loss of generality, that $q\in M_{X_L}$, and $L=\{x\in\bb{R}^n\mid x_1=0\}$. Let us also denote $\eta=(1,0,\ldots,0)$; it is the normal vector to $L$ such that $L^+=\{v\in \bb{R}^n\mid \langle x,\eta\rangle >0\}$. 
    Denote by $a,b$ two distinct points of $m_{X_L}(q)$. We will show that the direction $\eta$ is in the cone at infinity of $M_X$.

    Let us investigate the points above (in the $\eta$ direction) of the segment $]a,b[$. Take any $\lambda_1,\lambda_2>0$ such that $\lambda_1+\lambda_2=1$, $t>0$, and $x_t\in m(t\eta+\lambda_1a+\lambda_2b)$. Note here that since $x_t\notin L^+$ and $\lambda_1a+\lambda_2b\in L$ we have $\langle \eta, \lambda_1a+\lambda_2b-x_t \rangle \geq 0$.  Denote by $v_t:=t\eta+\lambda_1a+\lambda_2b-x_t/\|t\eta+\lambda_1a+\lambda_2b-x_t\|$, and define the directional reaching radius\footnote{The directional reaching radius is a concept defined and studied in \cite{BirbrairDenkowski}. It plays the role of a normal curvature for non-smooth sets; together with the reaching radius, defined in the same article and studied later in \cite{Bial1, BialBysDenk}, it characterises the point of the intersection $M_X\cap X$.} 
    \[r_{v}(x):=\sup\{\rho \geq 0\mid \bb{B}(\rho v+x,d(\rho v+x))\cap X = \emptyset\}.\]
    We claim that \[t\leq r_{v_t}(x_t)<\infty.\]

    If $x_t$ is not a unique closest point to $ t\eta+\lambda_1a+\lambda_2b$ in $X$ then the claim holds; $t = \|t\eta\| \leq \|t\eta+\lambda_1a+\lambda_2b-x_t\| =r_{v_t}(x_t)<\infty$ since $\langle \eta,\lambda_1a+\lambda_2b-x_t\rangle\geq 0$.

    If $x_t$ is unique, then the first inequality of the previous paragraph still holds and we need to justify only $r_{v_t}(x_t)<\infty$. Observe again that $r_{v_t}(x_t)=\infty$ means \[X\subset \{x\in \bb{R}^n\mid \langle t\eta + \lambda_1 a + \lambda_2 b - x_t,x-x_t\rangle\leq 0\}.\]
    Hence, for $x=a$ and $x=b$ we obtain
    \[\langle t\eta + \lambda_1 a + \lambda_2 b - x_t,a-x_t\rangle\leq 0 \text{ and } \langle t\eta + \lambda_1 a + \lambda_2 b - x_t,b-x_t\rangle\leq 0\]
    Multiplying the inequalities by $\lambda_1$ and $\lambda_2$ respectively and adding them together, we obtain
    \[\langle t\eta + \lambda_1 a + \lambda_2 b - x_t,\lambda_1a+\lambda_2 b-x_t\rangle\leq 0\]
    which can be rewritten as
    \[\langle t\eta, \lambda_1a+\lambda_2 b-x_t\rangle + \|\lambda_1 a + \lambda_2 b - x_t\|^2\leq 0.\]
    Since $\lambda_1 a + \lambda_2 b \neq x_t$, we get $\langle t\eta, \lambda_1a+\lambda_2 b-x_t\rangle <0 $, which is a contradiction.

    Since the value of $r_{v_t}(x_t)$ is finite, the point $r_{v_t}(x_t)v_t +x_t$ is well defined. By the standard argument, it belongs to the central set $C_X$.

    Observe now that for any $t>0$ the point $a$ is closest to $t\eta+a$. Since the multifunction $m$ is upper-semi-continuous, when $\lambda_1\to 1$ the points $x_t$ converge to $a$. In other words, 
    \[\forall t>0\forall \varepsilon>0\exists \lambda_t>0\forall x_{t}\in m(t\eta+\lambda_ta+(1-\lambda_t)b): \|x_{t}-a\|\leq\varepsilon. \]
    
    So choose $\varepsilon>0$ and $x_t\in m(t\eta + \lambda_ta+(1-\lambda_t)b)$ and denote $\Lambda_t:=\lambda_ta+(1-\lambda_t)b-x_t$. Now, the normal vector $v_t$ takes the form
    $v_t:=(t\eta + \Lambda_t)/\|t\eta + \Lambda_t\|$. We claim now that \[\|r_{v_t}(x_t)v_t +x_t\|\to\infty \text{ and  }(r_{v_t}(x_t)v_t +x_t)/\|r_{v_t}(x_t)v_t +x_t\|\to\eta \text{ when }t\to \infty \]

    The first convergence follows
    \[\|r_{v_t}(x_t)v_t +x_t\|\geq \langle r_{v_t}(x_t)v_t +x_t,\eta\rangle \geq \langle r_{v_t}(x_t)v_t,\eta\rangle -\varepsilon. \]
    Indeed, since $\|t\eta + \Lambda_t\|\leq \|t\eta\| + \|\Lambda_t\|\leq t + \Delta$ with $\Delta$ that can be picked universally for all $t$, we have
    \[\langle r_{v_t}(x_t)v_t,\eta\rangle 
    \geq
    t\left\langle\frac{t\eta + \lambda_ta+(1-\lambda_t)b-x_t}{\|t\eta + \Lambda_t\|},\eta\right\rangle 
    = 
    t\frac{\langle t\eta,\eta\rangle -\langle x_t,\eta\rangle}{\|t\eta + \Lambda_t\|} \geq \]
    \[\geq
    \frac{t^2-t\langle x_t,\eta\rangle}{t+\Delta}
    \geq
    \frac{t^2-t\varepsilon}{t+\Delta}\to \infty\]

    We will prove the second convergence by showing that 
    \[\left\langle\frac{r_{v_t}(x_t)v_t +x_t}{\|r_{v_t}(x_t)v_t +x_t\|},\eta\right\rangle\to 1.\]
    Firstly: 
    \[ \left\langle\frac{r_{v_t}(x_t)v_t +x_t}{\|r_{v_t}(x_t)v_t +x_t\|},\eta\right\rangle = \frac{\langle r_{v_t}(x_t)v_t,\eta\rangle}{{\|r_{v_t}(x_t)v_t +x_t\|}} +\frac{\langle x_t,\eta\rangle}{\|r_{v_t}(x_t)v_t +x_t\|}\]
    The second term converges to zero, while the first one is equal
    \[ \frac{\langle r_{v_t}(x_t)v_t,\eta\rangle}{{\|r_{v_t}(x_t)v_t +x_t\|}} = \frac{r_{v_t}(x_t)(t-\langle x_t,\eta\rangle)}{\|t\eta + \Lambda_t\|\cdot\| r_{v_t}(x_t)v_t +x_t\|} = \]
    \[\frac{1}{{\|\eta + \Lambda_t /t\|\cdot\|v_t +\frac{x_t}{r_{v_t}(x_t)}\|}}-\frac{\langle x_t,\eta\rangle}{\|t\eta + \Lambda_t\|\cdot\|v_t +\frac{x_t}{r_{v_t}(x_t)}\|} \]
    Since $\|\eta + \Lambda_t/t\|\to \|\eta\|=1$ and $\|v_t +\frac{x_t}{r_{v_t}(x_t)}\|\to\|v_t\|=1$  when $t\to \infty$, we obtain \[\left\langle\frac{r_{v_t}(x_t)v_t +x_t}{\|r_{v_t}(x_t)v_t +x_t\|},\eta\right\rangle \to 1.\] 
    Thus, we have found points $c_t := r_{v_t}(x_t)v_t +x_t\in C_X$ such that $c_t/\|c_t\|\to \eta$ when $t\to \infty$ and proved that the direction $\eta$ is in the cone at infinity of $M_X$.
    
      That is enough to end the proof, since for any $p$ in $L^+$ and a sequence of points $c_n\in C_X$ with $\|c_n\|\to\infty,\; c_n/\|c_n\|\to\eta$ we will have $d(c_n)>\|p-c_n\|$ for $n$ large enough. For any such $n$ the point $p$ belongs to $\bb{B}(c_n,d(c_n))$ so it is reconstructible.
\end{proof}

\begin{proposition}\label{Prop.Beyond}
    Let $X\subset\bb{R}^n$ be a closed set and let $p\in X^c$. Let $L$ be a supporting plane of $\overline{\conv X}$ such that $ X\cap L\neq \overline{\conv X}\cap L$. If $p\in L^+$ then $p$ is reconstructible.
\end{proposition}
\begin{proof}
    If $X \cap L$ is non convex, then the previous Proposition proves that $p$ is reconstructible. Assume without loss of generality that $L=\{x\in\bb{R}^n\mid x_1=0\}$ and that $X \cap L$ is convex. We will basically modify the proof of the previous Proposition to adopt it to new assumptions. Denote as before $\eta=(1,0,\ldots,0)$. Choose a point $w\in (\overline{\conv X} \cap L)\backslash(X \cap L)$. 
    Select $t>0$ and $x_t\in m(t\eta+w)$. 
    Note that since $x_t\notin L^+$ and $w\in L$ we have $\langle \eta, w-x_t \rangle \geq 0$.  
    Define the directional reaching radius:
    \[r_{v}(x):=\sup\{\rho \geq 0\mid \bb{B}(\rho v+x,d(\rho v+x))\cap X = \emptyset\}.\]
    We claim that \[t\leq r_{v_t}(x_t)<\infty.\]

    If $x_t$ is not a unique closest point to $ t\eta+w$ in $X$ then the claim holds $t = \|t\eta\| \leq \|t\eta+w-x_t\| =r_{v_t}(x_t)<\infty$ since $\langle \eta,w-x_t\rangle\geq 0$.

    If $x_t$ is unique, then the first inequality of the previous paragraph still holds and we need to justify only $r_{v_t}(x_t)<\infty$. Observe again that $r_{v_t}(x_t)=\infty$ means \[X\subset \{x\in \bb{R}^n\mid \langle t\eta + w - x_t,x-x_t\rangle\leq 0\}.\]

    Here the differences start.  We do not necessarily have $a,b\in L$ such that $w$ is their convex combination. However, since $w\in \overline{\conv X}$, we can find sequences of points $a_\nu,b_\nu\in X$ and scalars $\lambda_{\nu}\in(0,1)$  such that $w=\lim\lambda_\nu a_\nu+(1-\lambda_\nu)b_\nu$. We follow in the old footsteps with $x=a_\nu, b_\nu$. We obtain
    \[\langle t\eta + w - x_t,a_\nu-x_t\rangle\leq 0,\text{ and } \langle t\eta + w - x_t,b_\nu-x_t\rangle\leq 0\]
    Multiplying the inequalities by $\lambda_\nu$ and $1-\lambda_\nu$ respectively and adding them together, we obtain
    \[\langle t\eta + w - x_t,\lambda_\nu a + (1-\lambda_\nu) b-x_t\rangle\leq 0\]
    which can be rewritten as
    \[\langle t\eta, \lambda_\nu a + (1-\lambda_\nu) b-x_t\rangle + \langle w-x_t,\lambda_\nu a + (1-\lambda_\nu) b - x_t\rangle\leq 0.\]
    Since $\lambda_\nu a + (1-\lambda_\nu) b\to w$ when $\nu\to \infty$, and $w \neq x_t$ the second term tends to $\|w-x_t\|^2>0$. Consequently, we get as a limit $\langle t\eta, w-x_t\rangle <0 $, which is a contradiction.

    The proof of the previous proposition now approximates the points above some $\xi \in X \cap L$. However, we might not find any such $\xi$ under current assumptions. We need a new procedure. 

    Frankly,  we can simply go upward, from any point of $p\in X$ located in $(\overline{\conv X}\backslash X)\cap L$. Denote $\delta:=\min\{d(x,L)\mid x\in m(p)\}$, and call the point above $p$ to be $w_t = t\eta + p$. For the closest point $x_t\in m(w_t)$ the coordinates of $x_t=(x^1_t,\ldots,x^n_t)$ must satisfy 
    \[|x^1_t|<\sqrt{t^2+2\delta t+d(p)^2}-t \text{ and } \|(x_t^2,\ldots,x_t^n)\|\leq \sqrt{2t\delta+d(p)^2}.\] 
    We then also find that the vector $n_t$ given by 
    \[n_t:=w_t-x_t=((t+\delta)-x_t^1,p^2-x_t^2,\ldots,p^n-x_t^n)\]
    is normal to $X$ at $x_t$, and points at $w_t$. For $v_t=n_t/\|n_t\|$ the normalised expression $(r_{v_t}(x_t)v_t-x_t)/\|r_{v_t}(x_t)v_t-x_t\|$ will get steeper and steeper.

    Indeed, $r_{v_t}(x_t)\geq t\to\infty$ as we have seen a moment ago. Now, the fraction 
    \[\frac{r_{v_t}(x_t)v_t+x_t}{\|r_{v_t}(x_t)v_t+x_t\|}=\frac{r_{v_t}(x_t)v_t}{\|r_{v_t}(x_t)v_t+x_t\|}+\frac{x_t}{\|r_{v_t}(x_t)v_t+x_t\|}\]
     with the second term of the expression converging to zero when $t\to \infty$. The directional component of the first term is just scaled $n_t$, and we can clearly see that $n_t/t\to \eta$. So, the whole fraction $\frac{r_{v_t}(x_t)v_t+x_t}{\|r_{v_t}(x_t)v_t+x_t\|}$, being a vector on the unit sphere, must converge to $\eta$. Again, $\eta$ is a direction in the cone of $M_X$ at infinity. 

     By the same argument as at the end of the previous, the point $p$ belongs to some $\bb{B}(c,d(c))$ with $c\in C_X$, hence it is reconstructible.
\end{proof}

\begin{corollary}
    Let $X\subset \bb{R}^n$ be a closed set and let $p\in X^c$ be a point of the closure of the convex hull $\conv X$. Then $p$ is reconstructible.    
\end{corollary}
\begin{proof}
    Since $p\in\overline{\conv X}$ it is either a point of $\conv X$ and Proposition~\ref{Prop.ConvexHull} says it is reconstructible, or it is a point of $ \overline{\conv X}\backslash X$. In the latter case, any supporting hyper-plane $L$ of $\overline{\conv X}$ with $p\in L$ will raise a sequence of balls as in the proof of Proposition~\ref{Prop.Beyond} with $p\in \bb{B}(c_t,d(c_t)).$
\end{proof}

\begin{theorem}\label{Theorem}
    Let $X\subset\bb{R}^n$ be a closed set. Let $\mathcal{L}$ be a family of supporting hyperplanes of $\overline{\conv X}$ such that $X \cap L \neq \overline{\conv X} \cap L $. Then $X^c$ is  reconstructible if and only if it is a subset of a union $\overline{\conv X} \cup \bigcup_{L\in\mathcal L}L^+ $
\end{theorem}
\begin{proof}
    By previous Propositions, if $X^c$ is a subset of $\overline{\conv X}\cup \bigcup_{L\in\mathcal L}L^+ $ then all points in $X^c$ are reconstructible. It is so since $\overline{\conv X} \cap L$ is convex. So whenever $X \cap L$ is not convex, it cannot be equal to $\overline{\conv X} \cap L$. Let us briefly explain why no other point can be reconstructible.

    By definition, every reconstructible point $p\notin \conv X$ must be a point of some ball $\bb{B}(a,d(a))$ with $a\in M_X$. Pick a point $q\in m(a)$. Now, since the ball $\bb{B}(a,d(a))$ is convex, we can connect the points $p$ and $q$ with a segment $]p,q[\subset \bb{B}(a,d(a))$. The intersection of this segment with the closure of $\conv X$ forms a subinterval $[a,b[\,\subset\,]p,q[$. Now, we clearly see that for any supporting plane $L$ of $\overline{\conv X}$ with $a\in L$, the point $a$ is in $(L\cap \overline{\conv{X}})\backslash (X \cap L)$. When we pick $L$ to be a hyperplane that separates $p$ and $\overline{\conv X}$ we have the Theorem. 

\end{proof}

\begin{example}\label{Ex.Logarytm1}
We have all the possible preparation to investigate the example of 
\[X=\{(x,y)\in\bb{R}^2\mid y=\ln x\}\cup \{(-1,0)\} \text{ (see Fig.~\ref{Fig.Logarytm}).}\]

First, compute the convex hull of $X$. It is a subgraph of $y=\ln x$ together with the subgraph of the affine function tangent to $y=\ln x$ and passing through $(-1,0)$. The latter touches the graph at point $(\xi,\ln \xi)$ with $\xi = e^{W(1)+1}$ where $W$ denotes the Lambert $W$ function. Therefore, the convex hull is described by
\[\conv X = \{y\leq \ln x\}\cup \{y\leq \xi^{-1}(x + 1),x\in (-1,\xi]\}\cup\{(-1,0)\}.\] The supporting lines of $\overline{\conv{X}}$ are $x=-1$, $y=\xi^{-1}(x + 1)$, and every line tangent to the graph $y=\ln x$ at points $(x_0,y_0)$ with $x_0>\xi$. Only the first two have nonempty intersection with $\overline{\conv X}\backslash X$. So, the subset of $X^c$ that is reconstructible by $M_X$ is \[\left(\{y< \ln x\} \cup \{x<\xi\}\cup \{y>\xi^{-1}(x+1)\}\right)\backslash X.\] 

The reconstruction process does not provide information on the shape of $X$ in the region  
\[\Omega = \{ \ln x<y<e^{-W(1)-1}(x + 1),\;x>e^{W(1)+1}\}.\] 
As long as $X\cap \Omega = (\conv X)\cap \Omega$ any deformation of $X$ will not affect $M_X$ or the distance function $d$ restricted to $M_X$.

Note also that neither of the points of $M_X$ that lies above the graph of $y=\ln x$ (which, by the way, forms a smooth curve) has its distance realised to the right of $x=\xi$.
\end{example}


\begin{figure}
\centering
\includegraphics[width=0.7\textwidth]{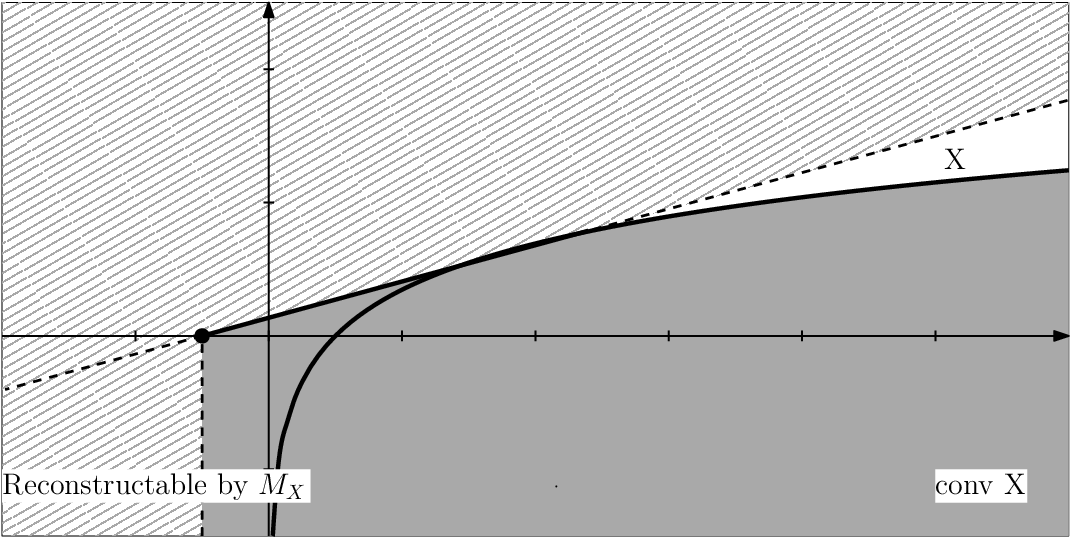}\label{Fig.Logarytm}
\caption{Example~\ref{Ex.Logarytm1} and \ref{Ex.Logarytm2}. The shaded region is the convex hull of $X$. During the reconstruction process with the medial axis, we get back points in $\overline{\conv X}$ together with ones in the striped region. The plain white region is not reached by balls centred at points of the medial axis. If we try to recreate the graph of the logarithm function by mapping the points of the medial axis through the multifunction $m$, we need to be careful. While traversing the connected component of $M_X$ that lies on the upper left we recreate only a part of the graph. The other connected component gives us the full graph of the logarithm.  }
\end{figure}

\begin{example}\label{Ex.Logarytm2}(The reconstruction of a smooth curve.) 
Let $X$ be the set from the Example~\ref{Ex.Logarytm1}. The medial axis of $X$ consists of two connected components $\gamma_1,\gamma_2$, which are $C^1-$smooth curves. Denote by $\gamma_i(t)$ their length parameterisations and let $T_i(t),N_i(t)$ denote, respectively, a unit tangent and a unit normal vector to $M_X$ at $\gamma_i(t)$. Then, considering a formula
\[x_i^\pm(t):=\gamma_i(t)-d((\gamma_i(t))d'((\gamma_i(t))T_i'(t)\pm d(\gamma_i(t))\sqrt{1-d'(\gamma(t))^2}N_i(t)\] we get a parametrisation of a part of $X$. (This construction is taken from Giblin \cite{GiblinReconstruction}.) From the analysis in Example~\ref{Ex.Logarytm1}, we can see that only one set of $x_i^\pm$ parametrises the whole curve $y=\ln x$. 
\end{example}

\section{Stability of the reconstructible points}\label{Sec.Stability}

The question of stability of the medial axis is prevalent in the literature. How does it apply to the notion of reconstructibility? Namely, How does the reconstructible part of $X^c$ behave when we change $X$ slightly?
In general, the change can be drastic. A straight line $L=\{y=0\}$ on the Euclidean plane has an empty medial axis, so the reconstructible part of $L^c$ is empty. If we puncture the line with an arbitrarily small gap $I_\varepsilon:=](-\varepsilon,0), (\varepsilon,0)[$, then the entire $(L\backslash I_\varepsilon)^c$ becomes reconstructible. 

Luckily, severe changes can appear only outside of $\conv X$. 
Indeed, Proposition~\ref{Prop.ConvexHull} asserts that $X^c\cap\conv X$ and $Y^c\cap\conv Y$ are reconstructable. If $X$ and $Y$ are close to each other (in the Hausdorff or Kuratowski sense), then their convex hulls are close as well. We therefore have a kind of  stability of a part of the reconstructable points. The only difference in reconstructability in the region of convex hulls of $X$ and $Y$ comes directly from the symmetric difference of $X$ and $Y$. However, remember that the complements of $X^c$ and $Y^c$ can still vary a lot even for close $X$ and $Y$.\footnote{ Take, for example, $X=\bb{R}\backslash\left((-1,1)\cup(k,k+1)\right)$ and $Y=\bb{R}\backslash ((-1,1))$. Even though $X$ and $Y$ each are contained in an $1-$offset of the other, their complements can be arbitrarily far from each other.} 

Next, we can use the Kuratowski convergence theorem to establish another result. Recall that for a family of sets $\{X_t\}_{t\in T }$ such that $X_t\to X_0$ in the Kuratowski sense when $t\to 0$ the medial axes satisfy an inclusion $M_{X_0}\subset \liminf_{t\to 0} M_{X_t}$ (see~\cite{BiaDenk_Kuratowski}). We have the following.

\begin{proposition}\label{Prop.Kuratowski}
    Let $\{X_t\}_{t\in T }$ be a family of closed subsets of $\bb{R}^n$ such that $\lim_{t\to 0} X_t = X$ in the Kuratowski sense and let $p\in X^c$ be reconstrucible. Then there exists a neighbourhood $T_0$ of $0\in T$ such that for any $t\in T_0$ the point $p\in X_t^c$ is reconstructible.
\end{proposition}
\begin{proof}
    By the definition of a reconstructible point, we can find $q\in M_X$ such that $p\in \bb{B}(q,d(q))$. Since $M_{X}\subset \liminf_{t\to 0} M_{X_t}$, we know that setting a neighbourhood $U$ of $q$ we can find $T_0$ such that for $t\in T_0$ every $M_{X_t}$ has a point in $q_t\in U$. Set $U$ to be a ball $B$ centred on $q$ and of radius $(d(q)-\|p-q\|)/2$. By shrinking $T_0$ we can assume further $|d_t(q_t)-d(q)| < (d(q)-\|p-q\|)/2$ as well. Then 
    \[
    \|p-q_t\|\leq \|q-q_t\|+\|p-q\|\leq d(q) - (d(q)-\|p-q\|)/2\leq d_t(q_t).
    \]
    So the point $p$ is reconstructible for all $X_t$ with $t\in T_o$ 
\end{proof}

\section{Other skeletal structures and settings}
\subsection[Lambda medial axis]{$\lambda-$medial axis}

What subset of $X^c$ can we reconstruct if we restrict the centres of balls in the union $\bigcup_{p\in M_X} \bb{B}(p,d(p))$ to the points $p\in M_\lambda(X)$ (see sec.~\ref{Subsect.LambdaMedialAxisDef})? Surely, the reconstructed part will be a subset of the reconstructible part of $X^c$ in the sense of definition~\ref{DefOfReconstructible}. 

Intuitively, new obstructions to reconstructibility will arise only in the bottlenecks of $X$. Bottlenecks can take the form of tunnels, sharp corners, or gates at the boundary of $\conv X$. Unfortunately, the proofs of Propositions \ref{Prop.ConvexHull} -- \ref{Prop.Beyond} do not discuss the diameter of $m(a)$. Actually, it is not immediately obvious how to bound the diameter of a ball encasing $m(a)$ in terms of $d(a)$ nor how this diameter behaves for $x$ close to $a$ (we only know that it is upper semi-continuous). It looks like the proof needs to be a bit more elaborate.

\begin{figure}
\centering
\includegraphics[width=0.7\textwidth]{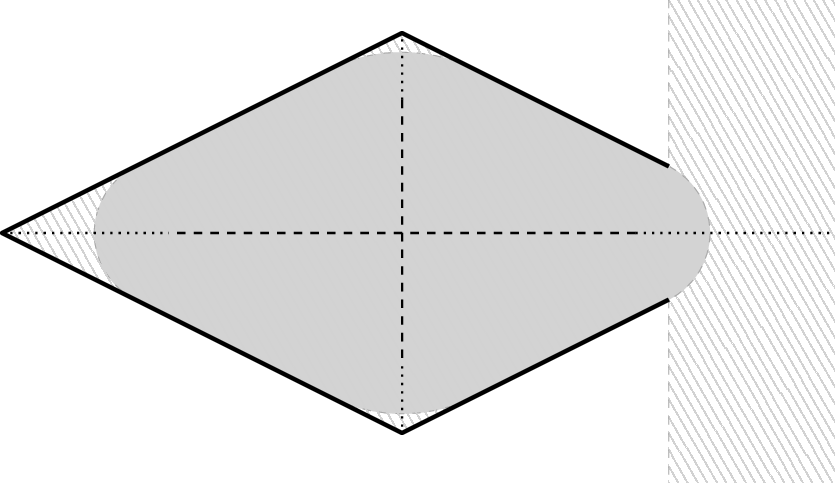}\label{Fig.LambdaSzkielet}
\caption{In general, the set of $\lambda-$reconstructible points forms a proper subset of reconstructible points. In example above, the dashed and dotted lines represent respectively the medial axis and the $\lambda-$medial axis of the shape. Only the grey area is $\lambda-$reconstructible, when $\lambda$ is set to be slightly larger than the gap on the right side of the shape. It forms a proper subset of the set of reconstructible points, depicted as the striped area.}
\end{figure}

In analogy to Definition~\ref{DefOfReconstructible}, we propose the following notion of $\lambda-$reconstructible points.
\begin{definition}\label{DefOfKReconstructible}
    Let $X\subset \bb{R}^n$ be a closed set and let $\lambda \geq 0$. A point $p\in X^c$ is said to be $\lambda-$ \textit{reconstructible} if $p\in \bigcup_{a\in M_\lambda (X)}\bb{B}(a,d(a))$.   
\end{definition}

Let $dist(A,B) := \inf \{\|a-b\|\mid a\in A, b\in B$. So far we have only the following weak version of proposition~\ref{Prop.Beyond}.

\begin{proposition}
    Let $X\subset\bb{R}^n$ be a closed set and let $L$ be a supporting plane of $ \conv X$ and denote $ X_L:= X\cap L$. Choose $\lambda>0$ and assume that there exists a point $q\in L$ such that $m_{X_L}(q)$ has at least two connected components and that for any two distinct connected components $A,B$ we have $dist(A,B)>2\lambda$. Then every point $p\in L^+$  is $\lambda-$reconstructible.
\end{proposition}
\begin{proof}

    As before, to simplify the exposition of the proof, assume, without loss of generality, that $q=0\in M_\lambda({X_L})$, $L=\{x\in\bb{R}^n\mid x_1=0\}$, and that $d(q,X\cap L) = 1$. Let us also denote $\eta=(1,0,\ldots,0)$;
    We will show that for any $T>0$ we can find points of $M_\lambda(X)$ in an unbounded cylinder $W_T=\{x\in\bb{R}^n\mid \langle x,\eta\rangle>T, \|x-\langle x,\eta\rangle\eta\|<1\}$.

    Denote
    \[
    X_\delta:=X\cap \{x\in\bb{R}^n\mid d(x,L)\leq\delta\}.
    \]
    We obviously have $X_L=X_0$. Moreover, since $X$ is closed, there is also $\lim_{\delta\to 0}X_\delta = X_0 $.

    Since $dist(A,B)>2\lambda$ for any two distinct connected components $A,B$ of $m_{X_L}(q)$, there exists a collection of open sets $U_\iota$, such that $m_{X_L}(q)\subset \bigcup U_\iota$ and \[\overline{U_i}\cap U_j=U_i\cap \overline{U_j}=\emptyset,\text{ when } i\neq j\] and  points \[x_\iota\in U_\iota\cap m_{X_L}(q).\]
     
    Actually, since the multifunction $m$ has compact values, we can assume that sets $U_\iota$ are still relatively far from each other, meaning $dist(\overline{U_i},\overline{U_j})>2\lambda$ for $i\neq j$.

    Since $X$ is closed, we can find $\delta$ such that $X_\delta\cap \overline{W_{-1}}$ is contained in the disjoint union $\bigcup U_\iota$. Then by choosing $t$ large enough we can ensure that the intervals $[x_i+t\eta,t\eta]$ are mapped by the multifunction $m$ onto $X_\delta\cap \overline{W_{-1}}\subset \bigcup U_\iota.$ 
     
    Now, traversing the segments $[x_i+t\eta,t\eta]$ and $[t\eta,x_j+t\eta]$ with $t>0$ large enough and $j\neq i$, we encounter a discontinuity of the multifunction $m$. The argument for which the discontinuity occurs is a point of the medial axis $M_X$. Moreover, since \[m(x_i+t\eta) = x_i\in U_i \text{ and } m(x_j+t\eta) = x_j \in U_j,\] we can say that for (at least) one of the discontinuity points, say $\xi$, the index $\iota$ for which $m(x)\cap U_\iota\neq \emptyset$ needs to change when $x$ is passing through $\xi$. Since $\lim_{x\to\xi}m(x)\subset m(\xi)$, the set $m(\xi)$ contains points from at least two distinct $U_\iota,U_\kappa$. The minimal radius of a ball that surrounds $m(\xi)$ will be at least equal to $dist(\overline{U_\iota},\overline{U_\kappa})/2>\lambda$.
\end{proof}

When it comes to $\lambda-$reconstructibility of the interior points of $X$, we conjecture.

\begin{conjecture}
    Let $X\subset\bb{R}^n$ be a closed set and let $p\in \conv X$. Choose $\lambda>0$ and assume that $d(p)>\lambda$. Then $p$ is $\lambda-$reconstrucitble.
\end{conjecture}

\subsection{Riemannian manifolds}
Do we have the same effect on the Riemannian manifolds? The change in space in which we let the set $X$ to live has consequences on the medial axis. First, the definition of the medial axis changes. It does not contain only the points with non-unique closest point anymore. It now collects the points with the non-unique shortest path to $X$ (consider a single point on a unit sphere to see the difference). Second, the rationale that involves operations on the scalar product must be transported to the tangent space, or geodesics. Last, one can also question the behaviour (and the definition) of supporting hyper-planes.

Mind, that one of the consequences of Theorem~\ref{Theorem} on a compact manifold $\mathcal{M}$ would be that the whole complement of a closed subset $X\subset\mathcal{M}$ is reconstructible. The author believes that the question is relevant in the context of satellite imagery.

However, the topic of the medial axis on riemannian manifolds is slightly neglected. When it comes to the reconstruction process, the author have found only~\cite{Wolter3D} that describes the process of set reconstruction on a $3-$dimensional manifold.

\section{Conclusion}
The ability of shape reconstruction is one of the revered traits of the medial axes. However, as we demonstrated at the beginning of the article, not all shapes are fully reconstructed. Therefore, we have defined and investigated the reconstructible points of the set complement $X^c$. The study shows that the reconstructible points of $X^c$ are precisely those of the convex hull of $X$ and of the half-spaces $L^+$, where $L$ is a supporting hyperplane of $\overline{X}$ and $X \cap L\neq \overline{\conv X}\cap L $. Some variants of the stability of the reconstructible points. 

We proposed the analogue of the reconstructible points for the $\lambda-$medial axis. However, the proofs of our theorems do not adapt easily to the $\lambda-$medial axis. For the $\lambda-$ medial axis, only a weak version of Proposition~\ref{Prop.Polprzestrzenie} was proved. We also propose a conjecture about the $\lambda-$reconstructible points in the convex hull of $X$. Lastly, we hint at the feasibility of the topic on Riemannian manifolds.

\section*{Acknowledgements}

The author was supported by a subvention to maintain research potential at AGH University of Kraków.
\bibliography{bibliography}

@article{Albano2,
	author = {Albano, Paolo},
	journal = {Nonlinear Anal.-Theor.},
	number = {C},
	pages = {398-405},
	title = {On the cut locus of closed sets},
	volume = {125},
	year = {2015},
    doi = {10.1016/j.na.2015.06.003}
}

@Inbook{Attali2009,
  author    = "Attali, Dominique
               and Boissonnat, Jean-Daniel
               and Edelsbrunner, Herbert",
  editor    = "M{\"o}ller, Torsten
               and Hamann, Bernd
               and Russell, Robert D.",
  title     = "Stability and Computation of Medial Axes - a State-of-the-Art Report",
  bookTitle = "Mathematical Foundations of Scientific Visualization, Computer Graphics, and Massive               Data Exploration",
  year      = "2009",
  publisher = "Springer Berlin Heidelberg",
  address   = "Berlin, Heidelberg",
  pages     = "109--125",
  doi       = "10.1007/b106657_6"}

@article{BiaDenk_Kuratowski,
  author    = "Adam Białożyt and Anna Denkowska and Maciej Denkowski",
  title     = "The Kuratowski convergence of medial axis and conflict sets",
  journal   = "Ann. Sc. Norm. Super. Pisa Cl. Sci",
  pages     = "21",
  year      = "2024",
  doi       = "10.2422/2036-2145.202310_002"
}

@article{Bial1,
  title     = {The tangent cone, the dimension and the frontier of the medial axis},
  author    = {Adam Białożyt},
  journal   = {Nonlinear Differential Equations and Applications NoDEA},
  year      = {2020},
  volume    = {30},
  pages     = {1-29},
  doi       = {10.1007/s00030-022-00833-9}
}

@article{BirbrairDenkowski,
  author    = {Lev Birbrair and Maciej Denkowski},
  title     = {Medial axis and Singularieties},
  journal   = {J. Geom. Anal.},
  volume    = {27},
  number    = {3},
  year      = {2017},
  pages     = {2339-2380},
  doi       = {10.1007/s12220-017-9763-x}
}

@misc{BirbrairDenkowskiErratum,
  title     = {Erratum to: Medial axis and singularities}, 
  author    = {Lev Birbrair and Maciej Denkowski},
  year      = {2017},
  eprint    = {1705.02788},
  archivePrefix={arXiv},
  primaryClass={math.MG},
  doi       = {10.48550/arXiv.1705.02788}
}

@Inbook{Chang2001,
  author    = "Chang, Y.-C.
              and Kao, J.-H.
              and Pinilla, J. M.
              and Dong, J.
              and Prinz, F. B.",
  editor    ="Kimura, Fumihiko",
  title     ="Medial Axis Transform (MAT) of General 2D Shapes and 3D Polyhedra for Engineering Applications",
  bookTitle ="Geometric Modelling: Theoretical and Computational Basis towards Advanced CAD Applications. IFIP TC5/WG5.2 Sixth International Workshop on Geometric Modelling December 7--9, 1998, Tokyo, Japan",
  year      ="2001",
  publisher ="Springer US",
  address   ="Boston, MA",
  pages     ="37--52",
  doi       ="10.1007/978-0-387-35490-3_3",
}

@article{ChazalLieutier,
  title     = {The $\lambda$-medial axis},
  author    = {Frederic Chazal and Andre Lieutier},
  journal   = {Graph. Models},
  year      = {2005},
  volume    = {67},
  pages     = {304-331},
  doi       = {10.1016/j.gmod.2005.01.002}
}

@article{Erdos,
  author    = "Erdös, Paul",
  fjournal  = "Bulletin of the American Mathematical Society",
  journal   = "Bull. Amer. Math. Soc.",
  number    = "10",
  pages     = "728--731",
  publisher = "American Mathematical Society",
  title     = "Some remarks on the measurability of certain sets",
  volume    = "51",
  year      = "1945",
  doi       = "10.1090/S0002-9904-1945-08429-8"
}

@article{Federer,
  author    = {Federer, Herbert},
  title     = {Curvature measures},
  journal   = {Trans. Amer. Math. Soc.},
  number    = {93},
  year      = {1959},
  pages     = {418-481},
  doi       = {10.2307/1993504}
}

@article{Fremlin, 
    title   = {Skeletons and central sets}, 
    volume  = {74}, 
    number  = {3}, 
    journal = {P. Lon. Math. Soc.}, 
    publisher={Cambridge University Press}, 
    author  = {Fremlin, David}, 
    year    = {1997}, 
    pages   = {701–720},
    doi     = {10.1112/S0024611597000233}
}

@article{GiblinReconstruction,
  author    = {P. J. Giblin and J.P. Warder},
  title     = {Reconstruction from Medial Representations},
  journal   = {The American Mathematical Monthly},
  volume    = {118},
  number    = {8},
  pages     = {712--725},
  year      = {2011},
  doi       = {10.4169/amer.math.monthly.118.08.712}
}

@inproceedings{Giesen,
  author    = {Giesen, Joachim and Miklos, Balint and Pauly, Mark and Wormser, Camille},
  title     = {The scale axis transform},
  year      = {2009},
  publisher = {Association for Computing Machinery},
  address   = {New York, NY, USA},
  doi       = {10.1145/1542362.1542388},
  booktitle = {Proceedings of the Twenty-Fifth Annual Symposium on Computational Geometry},
  pages     = {106–115},
  numpages  = {10}
}

@article{Houston,
title = {A Bose type formula for the internal medial axis of an embedded manifold},
journal = {Differential Geometry and its Applications},
volume = {27},
number = {2},
pages = {320-328},
year = {2009},
doi = {https://doi.org/10.1016/j.difgeo.2009.01.002},
author = {Kevin Houston and Martijn {van Manen}}
}

@book{RockafellarWets,
    title   = {Variational analysis.},
    author  = {Ralph Tyrrell Rockafellar and Roger Wets},
    year    = {2009},
    publisher= {Springer},
    address = {Berlin Heidelberg},
    doi     = {10.1007/978-3-642-02431-3}
}

@incollection{Blum,
  author    = {Harry Blum},
  booktitle = {Models for the Perception of Speech and Visual Form},
  title     = {{A} {T}ransformation for {E}xtracting {N}ew {D}escriptors of {S}hape},
  address   = {Cambridge},
  publisher = {MIT Press},
  year      = {1967},
  pages     = {362--380},
}

@INPROCEEDINGS{Wolter3D,
  author    = {Nass, Henning and Wolter, Franz-Erich and Thielhelm, Hannes and Dogan, Cem},
  booktitle = {2007 International Conference on Cyberworlds (CW'07)}, 
  title     = {Medial Axis (Inverse) Transform in Complete 3-Dimensional Riemannian Manifolds}, 
  year      = {2007},
  pages     = {386-395},
  doi       = {10.1109/CW.2007.55}
  }

@phdthesis{Wolter,
  author    = {Wolter, Franz-Erich},
  year      = {1985},
  school    = {Fechbereich Mathematik der Technischen Universitat Berlin},
  title     = {Cut loci in bordered and unbordered Riemannian manifolds}
}

@Inbook{Younes2019,
  author    = "Younes, Laurent",
  title     = "The Medial Axis",
  bookTitle = "Shapes and Diffeomorphisms",
  year      = "2019",
  publisher = "Springer Berlin Heidelberg",
  address   = "Berlin, Heidelberg",
  pages     = "57--72",
  doi       = "10.1007/978-3-662-58496-5_2"
}

\end{document}